\documentclass[10pt]{amsart}

\usepackage{amsfonts}
\usepackage{amssymb}
\theoremstyle{plain}
\newtheorem{thm}{Theorem}
\newtheorem{prop}[thm]{Proposition}
\newtheorem{cor}[thm]{Corollary}

\theoremstyle{definition}

\newcommand{\abs}[1]{\left\lvert #1 \right\rvert}

\newcommand{\norm}[1]{\left\| #1 \right\|}

\DeclareMathOperator{\spec}{spec}

\newcommand{\bbD}{\mathbb{D}}

\newcommand{\mattwo}[4]{\begin{bmatrix}#1 & #2 \\ #3 & #4 \end{bmatrix}}
\allowdisplaybreaks

\begin{document}

\title[Foguel]{A note on the norm and spectrum of a \\ Foguel operator}

\author[Mrinal Raghupathi]{Mrinal Raghupathi} 
\date{\today}
\address{Department of Mathematics, Vanderbilt University \\
  Nashville, Tennessee, 37240, U.S.A.}
\email{mrinal.raghupathi@vanderbilt.edu}
\urladdr{http://www.math.vanderbilt.edu/~mrinalr} \thanks{}
\begin{abstract}
  We present two ways to compute the norm of a Foguel operator. One of these is algebraic and the other makes use of the Schur complement. This
  gives a two simpler proof of a recent result of
  Garcia~\cite{garcia}. We also provide an extension of these results.
\thanks{The author was partially supported by a Young investigator award from Texas A{\&}M Univeristy in 2009.}
\end{abstract}

\maketitle

\section{Introduction}
Let $H$ be a Hilbert space. Given an isometry $V\in B(H)$ and an
operator $T\in B(H)$, the \textit{Foguel operator} with symbol $T$ is
defined by
\[
R_T = \begin{bmatrix}V^\ast & T \\ 0 & V\end{bmatrix}.
\]

In \cite{garcia}, a formula is given for the norm of a Foguel operator
$R_T$. The proof in~\cite{garcia} is based on an antilinear eigenvalue
problem and certain properties of complex symmetric matrices. In this
note, we give a more direct proof of this fact based on a direct
computation of the inverse. We also show how the formula for the norm
can be obtained as an application of the Schur complement.

Foguel operators have played a central role in counterexamples to
similarity conjectures. The most famous of these is Pisier's
counterexample~\cite{pisier} to the Halmos conjecture~\cite{halmos}.

\section{Spectrum and norm}

In this section we derive a relationship between the spectrum of the
operator $R_TR_T^*$ and $TT^*$. The result is due to
Garcia~\cite{garcia} but the proof given here is more direct. 

In the proof we make use of the fact that if $R$ is a selfadjoint
operator on a Hilbert space, then $R$ is invertible if and only if $R$
is right-invertible. To see this, suppose that $RS = I$. Then $S^*R^*
= S^*R = I$. Hence, $R$ has a left-inverse, and is invertible.

\begin{thm}[Garcia]\label{mainthm}
  If $\lambda>0$ and $\lambda\not=1$, then $\lambda\in
  \spec(R_TR_T^*)$ if and only if
  $\lambda^{-1}(\lambda-1)^2\in\spec(TT^*)$. The norm of $R_T$ is given by
\[
\norm{R_T} = \frac{\norm{T}+\sqrt{\norm{T}^2+4}}{2}.
\]
\end{thm}
\begin{proof}
  Suppose that $\lambda \not\in \spec(R_TR_T^\ast )$. Since
  $R_TR_T^*-\lambda I$ is invertible there exists a selfadjoint
  operator
\[
S = \begin{bmatrix} A & X \\ X^\ast & B \end{bmatrix}
\] 
such that $R_TR_T^\ast S = \lambda S + I$. We have
\[
R_TR_T^\ast = 
\mattwo{V^\ast V + TT^\ast}{TV^\ast}{VT^\ast}{VV^\ast} =
\mattwo{I+TT^\ast}{TV^\ast}{VT^*}{VV^*}.
\]
Writing out the entries of the operator matrix equation $R_TR_T^*S =
\lambda S + I$ we get
\begin{align}
\label{aa}(I+TT^*)A + TV^*X^* &= \lambda A + I\\
\label{ab}(I+TT^\ast)X + TV^*B &= \lambda X\\ 
\label{ba}VT^*A + VV^*X^* &= \lambda X^*\\
\label{bb}VT^*X+VV^*B &= \lambda B + I.
\end{align}
Using the fact that $V^*V = I$ and multiplying equation \eqref{ba} on
the left by $V^*$ we get $T^*A+V^*X^* = \lambda V^*X^*$, which gives
\begin{equation}
\label{eq1}T^*A = (\lambda-1)V^*X^*.
\end{equation}
Multiplying \eqref{aa} by $(\lambda-1)$ and substituting the expression from \eqref{eq1} we get
\begin{equation}
(\lambda-1)(I+TT^*)A + TT^\ast A = (\lambda-1)(\lambda A  + I).
\end{equation}
Rearranging and simplifying this last equation we get
\[
((\lambda-1)I+I)TT^*A+((\lambda-I)+\lambda(\lambda-1))A = (\lambda
-1)I
\] 
or
\[
(\lambda TT^*-(\lambda-I)^2)A = (\lambda-1)I.
\] 

Since we have assumed that $\lambda\not=0,1$ we get
\begin{equation}
\label{eq4}
A = \frac{\lambda}{\lambda-1}\left(TT^* -
  \frac{(\lambda-1)^2}{\lambda}I\right)^{-1}.
\end{equation} 
Hence $\lambda^{-1}(\lambda-1)^2\not\in\spec(TT^*)$.

The operators $X$ and $B$ can also be computed explicitly. The choice
of $X = \frac{1}{\lambda-1} ATV^\ast$ satisfies
\eqref{eq1}. Multiplying \eqref{bb} by $V^*$ we get $T^*X + V^*B =
\lambda V^*B + V^*$ and so
\begin{equation}
\label{eq2}(\lambda-1)V^*B = T^*X - V^*. 
\end{equation}
The choice 
\begin{equation}
\label{eq3}B = (\lambda-1)^{-1}(VT^*X - VV^*)
\end{equation}
is a solution of \eqref{eq2}.

Now consider the case where $\mu\not\in \spec(TT^*)$ and assume that
$\mu\not=0$. The equation $\mu = (\lambda-1)^2/\lambda$ has two
positive solutions, choose $\lambda$ to be either of these
solutions. We choose $A,X,B$ as in the equations above.  It is a
routine, but straightforward, calculation to show that this choice of
$A,B,X$ satisfies \eqref{aa}--\eqref{bb}.

From the relationship between the spectrum of $R_TR_T^*$ and the
spectrum of $TT^*$ we get that
\begin{equation}
\norm{T}^2 = \sup\{\lambda^{-1}(\lambda-1)^2 \,:\, \lambda\in
\spec(R_TR_T^*)\}.
\end{equation}

The equation $\norm{T}^2 = \lambda^{-1}(\lambda-1)^2$ has two
solutions of the form $\lambda,\lambda^{-1}$. The function $f(\lambda)
= \lambda^{-1}(\lambda-1)^2$ is increasing for $\lambda\geq 1$ and
$\norm{T}^2$ is the maximum value of the function $f$ on the set
$\spec(R_TR_T^*)$. Since, $\norm{R_T}\geq \norm{V} = 1$ we see that
\begin{equation}
\label{eq6}\norm{T}^2 = (\norm{R_T}^2 - 1)/\norm{R_T}.
\end{equation}
This gives,
\begin{equation}
\norm{R_T} = \frac{\norm{T} + \sqrt{\norm{T}^2+4}}{2}.
\end{equation}
\end{proof}

The proof of Theorem~\ref{mainthm} actually gives us a little more
information about the invertibility of $R_TR_T^*$ and $TT^*$.

\begin{prop} 
  The operator $R_TR_T^*$ is invertible if and only if $V$ is
  unitary. The operator $R_TR_T^*-I$ is invertible if and only if $T$
  is invertible.
\end{prop}
\begin{proof}
  From \eqref{bb} we have, $V(T^*X + V^*B) = I$ and so $VY = I$ for
  some $Y$. Hence, $Y = (VV^*)Y = V^\ast( VY ) = V^\ast$. Which proves
  that $VV^\ast =I$. 

  On the other hand, if $V$ is unitary, then
  $\mattwo{V}{-VTV^*}{0}{V^*}$ is the inverse of $R_T{R_T}^*$.

  Now consider the case where $R_TR_T^\ast - I$ is invertible. In this
  case from \eqref{aa} we get $TT^*A + TV^*X^* = I$, which gives, $TY
  = I$ for some $Y$. Multiply equation \eqref{bb} by $V^*$ we get
  $T^*X = V^*$. Now multiply by $V$ to get $T^*XV = I$. Hence, $ZT =
  I$ for some $Z$ and so $T$ is invertible. However, we now get from
  \eqref{eq1} that $T^\ast A = 0$ and so $A=0$. In this case, the
  inverse of $R_TR_T^\ast -I$ is given by $\mattwo{0}{X}{X^\ast}{B}$,
  where $X = (T^{-1})^*V^*$ and $B = -I$.

  Conversely, if $T$ is invertible, then there exists $Z$ such that
  $ZT = TZ = I$ and we can write down the inverse of $R_TR_T^*-I$ as above.
\end{proof}

As a final note in this section we point out that we can strengthen the power-bounded
result in~\cite{garcia} to the case of polynomials in $R_T$. 

\begin{prop}\label{contractive}
  Let $A$ be a contraction and let $T$ be an operator on $H$. Let $R
  = \begin{bmatrix} A^* & T \\ 0 & A \end{bmatrix}$. We have,
\[\norm{R} \leq \frac{\norm{T}+\sqrt{\norm{T}^2+4}}{2}\]
\end{prop}
\begin{proof}
  Let $V_A$ denote the usual isometric Sz.-Nagy dilation of $A$, that is, 
  \[V_A = \begin{bmatrix} A & (I-AA^*)^{1/2} \\ (I-A^*A)^{1/2} &
    A^*\end{bmatrix}.\] Let $\tilde{T}$ denote the $2\times 2$
  operator matrix that has $T$ in its $(1,2)$ entry and is 0
  otherwise. Let $W = \begin{bmatrix} V_{A^*} & \tilde{T} \\ 0 &
    V_A\end{bmatrix}$. Note that $V_{A^*} = V_A^*$. Since $V_A$ is an
  isometry, $W$ is a Foguel operator with symbol $\tilde{T}$ and so
  $\norm{W} = \frac{1}{2}\big(\norm{T} + (\norm{T}^2+4)^{1/2}\big)$, since
  $\|\tilde{T}\| = \norm{T}$. If we view $W$ as a $4\times 4$
  operator matrix, then the operator $R$ is the compression of $W$ to
  the first and fourth row and column. Hence, $\norm{R} \leq \norm{W}$.
\end{proof}

\begin{cor}
  Let $p(z) = a_0+a_1z+\cdots+a_mz^m$ be a polynomial. Let
  $\norm{p}_\infty:= \sup_{z\in \bbD} \abs{p(z)}$ and let $\tilde{p}(z) = \abs{a_0}+\abs{a_1}z+\cdots + \abs{a_m}z^m$. Let $A$ and $T$ be
  operators on $H$ with $\norm{A}\leq 1$ and let
  \[R = \begin{bmatrix} A^* & T \\ 0 & A \end{bmatrix}.\]
  We have, 
  \[\norm{p(R)} \leq \frac{\norm{\tilde{p}'}_\infty\norm{T} +
    \sqrt{\norm{\tilde{p}'}_\infty^2\norm{T}^2+4}}{2}\]
\end{cor}
\begin{proof}
  First note that we can assume that $\norm{p}_\infty \leq 1$. A
  simple computation shows that
  \[R^n = \begin{bmatrix} (A^*)^n & \sum_{j=0}^{n-1} (A^*)^j T
    A^{n-1-j} \\ 0 & A^n \end{bmatrix}\] Let us denote the operator in
  the upper right corner of the above matrix by $D_n(A,T)$ for $n\geq 1$.

We have,
\[p(R) = \mattwo{p(A)^*}{\sum_{j=1}^m a_jD_j(A,T)}{0}{A^n}.\] Since $A$
is a contraction, von-Neumann's inequality tells us that $\norm{p(A)}
\leq \norm{p}_\infty \leq 1$. Now,
\begin{align*}
  \norm{\sum_{j=1}^m a_j D_j(A,T)} \leq &
  \sum_{j=1}^m \abs{a_j}\norm{D_j(A,T)} \leq \sum_{j=1}^m \abs{a_j}\sum_{i=0}^{j-1} \norm{(A^*)^i T A^{j-1-i}}\\
  \leq &\sum_{j=1}^m \abs{a_j}\sum_{i=0}^{j-1} \norm{A}^{j-1} \norm{T}
  \leq \norm{T} \sum_{j=1}^m j\abs{a_j}\norm{A}^{j-1} \\
= &  \norm{T}\tilde{p}'(\norm{A}) \leq \norm{T}\norm{\tilde{p}'}_\infty 
\end{align*}
The result now follows from Proposition~\ref{contractive}.
\end{proof}

An application of the previous proposition with $p(z)= z^n$ gives the
result on the norm of $R_T^n$ obtained in~\cite{garcia}. We have $p =
\tilde{p}$ and $\norm{p'}_\infty = n$. Hence, we obtain the norm estimate
\[\norm{R_T^n} \leq \frac{n\norm{T}+\sqrt{n^2\norm{T}^2+4}}{2}.\]

\section{A positivity proof of the norm}
In this section we present a short proof of the norm equality using the
Schur complement. 

We begin by giving a brief description of the Schur complement. Let $R
= \mattwo{P}{X}{X^*}{Q}$, with $P\geq 0,Q>0$. The operator $R$ is positive if and only if the matrix $P
- XQ^{-1}X^*\geq 0$. The proof of this fact follows by first
conjugating $R$ by the positive matrix $I \oplus Q^{-1/2}$ which gives
$R' = \mattwo{P}{XQ^{-1/2}}{Q^{-1/2}X^\ast }{I}$, and then using the
fact that $\mattwo{P}{Y}{Y^*}{I}\geq 0$ if and only if $P - YY^*\geq
0$.

For an operator $T\in B(H)$, $\norm{T} = \inf\{M\,:\, M^2I-TT^*> 0\}$. 

Consider the Foguel operator $R_T$. We know that $\norm{R_T}\geq
1$. Let $M>1$. We have
\begin{align}
  M^2I - R_TR_T^* &= \mattwo{M^2I}{0}{0}{M^2I}-\mattwo{I+TT^*}{TV^*}{VT^*}{VV^*}\\
  & = \mattwo{(M^2-1)I - TT^*}{-TV^*}{-VT^*}{M^2I - VV^*}
\end{align}
By applying the Schur complement we see that this operator is positive
if and only if
\[(M^2-1)I - TT^* - TV^*(M^2I - VV^*)^{-1}VT^* \geq 0.\] Since $M>1$
and $V$ is a contraction we can expand $(M^2I-VV^*)^{-1}$ in its
Neumann series as $M^{-2}\sum_{j=0}^\infty (VV^*)^j/(M^{2j})$. Hence,
\begin{align}
  TV^*(M^2I-VV*)^{-1}VT^* & = T\left(M^{-2}\sum_{j=0}^\infty \frac{V(VV^*)^jV^*}{M^{2j}}\right)T^*\\
  & = M^{-2}TT^* \sum_{j=0}^\infty M^{-2j} = (M^2-1)^{-1}TT^*.
\end{align}

Hence, the positivity condition is
\[(M^2-1)I - TT^* - (M^2-1)^{-1}TT^* = (M^2-1)I -
M^2(M^2-1)^{-1}TT^*\geq 0.\] This happens if and only if
\[TT^* \leq M^{-2}(M^2-1)^2,\] which happens if and only if
\[\norm{T}\leq M^{-1}(M^2-1).\]
Hence, $\norm{T} = \norm{R_T}^{-1}(\norm{R_T}^2-1)$, which is the same
relation obtained in \eqref{eq6}.

\end{document}